\documentclass[11pt,reqno]{amsart}
\usepackage{amscd,amsmath,amsopn,amssymb,amsthm,multicol}
\usepackage{tikz,subdepth,anysize,verbatim,ifthen,xargs,colortbl,float}
\usepackage{longtable,mathtools}
\usepackage[english]{babel}
\usepackage[utf8]{inputenc}
\everymath=\expandafter{\the\everymath\displaystyle}

\usetikzlibrary{decorations.markings,arrows,arrows.meta,bending,calc}
\tikzset{>={Stealth[scale=1.5, bend]}}
\theoremstyle{plain}
\newtheorem{theorem}{Theorem}
\newtheorem{prop}{Proposition}
\newtheorem{lemma}{Lemma}

\newtheorem{rk}{Remark}
\theoremstyle{definition}

\newcommand\com[1]{}

\newcommand\f{\varphi}
\newcommand\g{{\gamma}}

\newcommand\op[1]{\mathop{\rm #1}\nolimits}
\newcommand\p{\partial}

\newcommand\R{{\mathbb R}}

\newcommand\ff{{\boldsymbol f}}
\newcommand\bh{{\boldsymbol h}}
\newcommand\bx{{\boldsymbol x}}
\newcommand\by{{\boldsymbol y}}

\makeatletter
    \def\@@and{}
\makeatother

\begin{document}

\title{Variationality of Conformal Geodesics in dimension 3}

\author{Boris Kruglikov}
\address{Department of Mathematics and Statistics, UiT the Arctic University of Norway, Troms\o\ 9037, Norway.
\ E-mail: {\tt boris.kruglikov@uit.no}. }

\author{Vladimir S.\ Matveev}
\address{Institute of Mathematics, Friedrich-Schiller-Universität, 07737 Jena, Germany.\newline
\ E-mail: {\tt vladimir.matveev@uni-jena.de}.}

\author{Wijnand Steneker}
\address{Department of Mathematics and Statistics, UiT the Arctic University of Norway, Troms\o\ 9037, Norway.
\ E-mail: {\tt wijnand.s.steneker@uit.no}. }

 \begin{abstract}
Conformal geodesics form an invariantly defined family of unparametrized curves in a conformal manifold
generalizing unparametrized geodesics/paths of projective connections. The equation describing them is of third order, 
and it was an open problem whether they are given by an Euler--Lagrange equation. 
In dimension 3 (the simplest, but most important from the viewpoint of physical applications) we demonstrate that
the equation for unparametrized conformal geodesics is variational.
 \end{abstract}

\maketitle


Conformal geodesics naturally arise in general relativity as a tool in studying conformal infinities \cite{FS,AK}. 
Variationality for classical geodesics is well known, and is a fundamental tool in their study.
Therefore it is important to understand and develop variationality for conformal geodesics.

\section{Conformal geodesics}\label{S1}

Let $g$ be a pseudo-Riemannian metric on a manifold $M$, $\nabla$ its Levi-Civita connection,
$\op{Ric}_g$ the Ricci tensor with traceless part $\op{Ric}_g^0$ and trace $R_g$ (scalar curvature).

Conformal geodesics of the conformal structure $[g]$ are given by the following differential equation \cite{Y,S,FS,GST}
 \begin{equation}\label{BE0}
\nabla_ua=\frac{3g(u,a)}{|u|^2}a- \frac{3|a|^2}{2|u|^2}u +|u|^2P^\sharp u-2P(u,u)u,
 \end{equation}
where $u=\dot{x}$ is the velocity field along non-null curve $x:I\to M$, $a=\nabla_uu$ is the acceleration,
the norms are defined as $|u|^2=g(u,u)$, the scalar product as $v\cdot w=g(v,w)$,
and $P$ is the Schouten tensor
 $$ 
P=\frac1{n-2}\Bigl(\op{Ric}_g-\frac1{2(n-1)}R_g\cdot g\Bigr)
=\frac1{n-2}\op{Ric}^0_g+\frac1{2n(n-1)}R_g\cdot g.
 $$

Following Bailey--Eastwood \cite{BE} we split \eqref{BE0} into its tangential part
 \begin{equation}\label{BE1}
u\cdot\nabla_ua=3\frac{(u\cdot a)^2}{|u|^2}-\frac32|a|^2-|u|^2P(u,u)
 \end{equation}
and the normal part
  \begin{equation}\label{BE2}
u\times \nabla_ua=3\frac{u\cdot a}{|u|^2}u\times a-|u|^2P^\sharp u\times u.
 \end{equation}
(One may understand $\times$ as the wedge operation, but in 3D it can be identified via Hodge star of $g$
and musical isomorphisms with the operation of vector product on $TM$.)
 
Let us summarize the following important properties of conformal geodesics \cite{S,BE,T}:

1$^\circ.$ 
Equation \eqref{BE0} is covariant, i.e.\ it does not depend on a choice of coordinates on $M$, and so transforms naturally under 
diffeomorphisms. Next, we observe that a change of a representative $g$ of the conformal structure $[g]$ leads to the same differential equation. Consequently, \eqref{BE0} is conformally invariant: if $\phi:(M,g)\to(\bar{M},\bar{g})$ is a conformal diffeomorphism, then it transforms the equation for conformal geodesics of $g$ to that of $\bar{g}$. 
Equivalently, $\phi$ maps conformal geodesics of $g$ to those of $\bar{g}$. 

$2^\circ.$ 
A curve is a conformal geodesic if and only if there exists a metric $g\in[g]$ such that the curve is an 
affinely parametrized non-null geodesic and $P^\sharp u=0$.
Note that for a generic representative $g$ of the conformal structure $[g]$ the Schouten tensor is nondegenerate, and so 
the equation $P^\sharp u=0$ has no nonzero solutions, but changing
$g\to\bar{g}=e^{2f}g$ results in $P\mapsto\bar{P}=P-\op{Hess}(f)+df\otimes df-\frac12|df^2|g$, $\op{Hess}(f)=\nabla df$,
and we can impose the nonlinear Monge-Ampere equation $\op{det}(\bar{P})=0$.

$3^\circ.$
Each solution of \eqref{BE0} enters with a canonical projective parameter, i.e.\ a 3-dimensional space of parametrizations
related by M\"obius transformations, and such reparametrizations do not change the equations.
Arbitrary change of parameter $t\mapsto h(t)$ along the curve changes \eqref{BE1} by a multiple of the Schwarzian derivative 
of $h$ and keeps \eqref{BE2} invariant. In fact, choosing a proper parameter one can always satisfy \eqref{BE1}, so the 
unparametrized version of the conformal geodesic equation is \eqref{BE2}. 

\medskip 

A constrained version of the parametrized conformal geodesic equation is based on normalization 
of the length parameter $|u|^2=1$ (depending on a metric $g$ in the conformal class;
in pseudo-Riemannian case this should be changed to $\pm1$), 
which implies $u\cdot a=0$ and furthermore $u\cdot\nabla_ua+|a|^2=0$.
Combining this with \eqref{BE1} yields $P(u,u)=-\frac12|a|^2$ and hence \eqref{BE0} with such constraints can be written as
 \begin{equation}\label{BE3}
\nabla_ua=-\frac12|a|^2u+P^\sharp u.
 \end{equation}

Variationality of the equations for conformal geodesics was discussed in the literature \cite{BE,SZ,DK} but
these results do not give an answer to the classical question whether the equations of conformal geodesics
are precisely the Euler--Lagrange (EL) equations of some Lagrangian, which we would like to address here. 

In what follows we restrict to the case $n=\dim M=3$ and we can assume without loss of generality that $|u|^2>0$.
As discussed above, and as is customary for projective structures, it is natural to consider unparametrized conformal geodesics. 
Our main result is
 \begin{theorem}\label{Thm}
The equation for unparametrized conformal geodesics is variational with Lagrangian 
$L=V/A^2$, where $V=\op{Vol}_g(u,a,\nabla_ua)$ and $A=\op{Area}_g(u,a)$; see \eqref{LL}.
The corresponding action is $\int L\,dt=\int\tau\,ds$, where $\tau$ is the torsion of the curve and $s$ is the arc length parameter.
 \end{theorem}
In the case of parametrized conformal geodesics the situation is opposite: the corresponding equation is not
variational. This peculiar behavior is not characteristic for classical variational problems, but the differential equations of
odd order are also non-standard and less studied.

 \begin{rk}
When our paper was ready and proofread, we discovered preprint \cite{TM} by T.\,Marugame that appeared in arXiv 
a week prior to our submission. Theorem 1.2 there states the same as our Theorem \ref{Thm}. 
The methods however are different, and the works independent. 

It is worth noticing that \cite{TM} relates conformal geodesics to chains in CR geometry. Previously those were shown
to be variational via a Kropina metric in \cite{CMMM}. Chains are complex analogs of geodesics; for more general
paths variationality is a restrictive property, see \cite{KM} and referenes therein.
 \end{rk}
 
The structure of the paper is as follows. After a short preliminary on the inverse variational problem,
we discuss the conformally flat case, in which case the problem was treated before in \cite{BF1,BF2}.
We revisit this with an independent computation and note that the known explicit Lagrangian is not covariant. 
Then we pass to the general case with a proper generalization of the Lagrangian. 
At the end of the paper we discuss conformal invariance of the Lagrangian.

\section{The inverse variational problem}\label{S2}

We restrict our discussion to the case of autonomous ODEs. Such systems of order $k$ on vector variable $\bx=(x^0,\dots,x^n)$ 
have the following form for some $\ff=(f^0,\dots,f^n)$ where $\bx^{(i)}=\tfrac{d^i}{dt^i}\bx(t)$:
 \begin{equation}\label{ODE}
\ff(\bx,\bx',\dots,\bx^{(k)})=0.
 \end{equation}
For parametrized solutions the inverse problem, namely if $\ff$ of \eqref{ODE} is given by the Euler--Lagrange operator 
$\tfrac{\delta L}{\delta\bx}$ for some Lagrangian $L=L(\bx,\bx',\dots)$, can be solved by the so-called Helmholtz conditions,
which are well-studied in the case of second-order equations $k=2$.
We refer to numerous literature on variational problems, see for instance \cite{An,G,Krp,Do} and references therein.

The inverse problem for a {\em differential equation\/} $\ff=0$ (contrary to that of a {\em differential operator} $\ff$),
i.e.\ whether \eqref{ODE} is an Euler--Lagrange equation for some $L$,
is much more difficult, and in general is not solved even for $k=2$. It is even more difficult in the case of unparametrized 
solutions (which for $k=2$ are the so-called paths). Namely, choose $x=x^0$ as a parameter along the solution-curve
$\by=\by(x)$ for $\by=(y^1,\dots,y^n)$, $y^j=x^j$ for $j=1,\dots,n$, and rewrite equation \eqref{ODE} as
a non-autonomous system 
 \begin{equation}\label{ODE'}
\by^{(k)}=\bh(x,\by,\by',\dots,\by^{(k-1)}),
 \end{equation}
where now $\by^{(k)}=\tfrac{d^k}{d x^k}\by(x)$ and $\bh=(h^1,\dots,h^n)$.

In the case of second-order $k=2$ the following is known. In dimension $n+1=2$	
every path structure is variational as was demonstrated by N.\,Sonin in 1886 \cite{So}. 
Dimension $n+1=3$ was investigated in detail by Davis and Douglas \cite{Dav,D}. 
Higher dimensions were also investigated in the literature, and indeed a generic ODE is not variational;
we refer to \cite{KM} for a discussion.

Our problem with the third order ODEs is more sophisticated.
Below we consider time-independent Lagrangian $L$ since the equations are autonomous.
The Euler--Lagrange operator $L\mapsto\op{EL}(L)$ has components
 $$
\frac{\delta L}{\delta x^k}=
\frac{\p L}{\p x^k}-\frac{d}{dt}\frac{\p L}{\p\dot{x}^k}+\frac{d^2}{dt^2}\frac{\p L}{\p\ddot{x}^k}-
\frac{d^3}{dt^3}\frac{\p L}{\p\dddot{x}^k}+\dots
 $$
where (in the autonomous case) $\frac{d}{dt}=\dot{x}^i\p_{x^i}+\ddot{x}^i\p_{\dot{x}^i}+\dddot{x}^i\p_{\ddot{x}^i}+\dots$
is the operator of total derivative. 
 
Suppose that $L$ is homogeneous in $\bx^{(i)}$ of degree $m_i$. Then the total degree is defined to be $|m|=\sum im_i$.

 \begin{prop}\label{P1}
Let a reparametrization invariant ODE system of 3rd order be variational. 
Then it is variational in the class of Lagrangians of order 2: 
there exists a function $L=L(\bx,\dot{\bx},\ddot{\bx})$ homogeneous of total degree $1$ and affine in 2-jets $\ddot{\bx}$,
whose extremals are precisely the curves of the ODE. 
 \end{prop}
 
 \begin{proof}
By the Vainberg--Tonti  formula \cite{Krp}, we may assume without loss of generality that the Lagrangian 
of the variational 3rd order system is of 3rd order: $\hat{L}=\hat{L}(\bx,\dot{\bx},\ddot{\bx},\dddot{\bx})$.
In the corresponding Euler--Lagrange equation $\op{EL}(\hat{L})$ the 6th order terms come from 
$\frac{d^3}{dt^3} \frac{\p\hat{L}}{\p\dddot{x}^i}$.
Hence the Lagrangian has the following form:
  \begin{equation}\label{eq:EL1}
\hat{L}(\bx,\dot{\bx},\ddot{\bx},\dddot{\bx})=F(\bx,\dot{\bx},\ddot{\bx})
+\sum_s\dddot{\bx}^s \lambda_s(\bx,\dot{\bx},\ddot{\bx}).
 \end{equation}
Vanishing of the 5th order jets in $\op{EL}(\hat{L})$ is equivalent to the system
 \begin{equation}\label{eq:EL2}
\frac{\p\lambda_s }{\p\ddot{x}^i} - \frac{\p\lambda_i }{\p\ddot{x}^s} =0.
  \end{equation}
Then there exists a function $\Lambda(\bx,\dot{\bx},\ddot{\bx})$ such that $\lambda_s= \tfrac{\p\Lambda}{\p\ddot{x}^s}$ 
implying $\sum_s \dddot{\bx}^s\tfrac{\p\Lambda}{\p\ddot{x}^s}= \tfrac{d}{dt} \Lambda 
 - \sum_s \dot{x}^s \tfrac{\p\Lambda}{\p x^s}  - \sum_s \ddot{x}^s\tfrac{\p\Lambda}{\p\dot{x}^s}$.

Since addition of the total derivative $-\tfrac{d}{dt}\Lambda$ to a Lagrangian does not change
the complete variation, the Euler--Lagrange equation with Lagrangian  \eqref{eq:EL1}
coincides with the Euler--Lagrange equation for
 $$
\tilde L= F(\bx,\dot{\bx},\ddot{\bx}) - \sum_s \dot{x}^s \tfrac{\p\Lambda}{\p x^s}  - \sum_s \ddot{x}^s\tfrac{\p\Lambda}{\p\dot{x}^s},
 $$
which is of the 2nd order. Furthermore, since the 4th order terms in $\op{EL}(\tilde{L})$ vanish, the Lagrangian
$\tilde{L}$ must be affine in $\ddot{\bx}$. Finally,
since by our assumptions a reparameterization of a solution $\bx(t)\mapsto\bx(\tau(t))$ is also a solution,
the Lagrangian $\tilde L$ is homogeneous of total degree 1. 
 \end{proof}

Thus we assume $L$ to be affine in acceleration, namely
 \begin{equation}\label{Laff}
L=\sum_{k=1}^3 \g_k(x,u)a^k+\g_0(x,u).
 \end{equation}
This actually gives third order Euler--Lagrange equations, yet they cannot be \eqref{BE0}. 

 \begin{prop}
Equation \eqref{BE0} of parametrized conformal geodesics in 3D is not variational. 
 \end{prop}
 
 \begin{proof}
Denoting by dots the terms of jet-order 2 in $\op{EL}(L)$ for $L$ given by \eqref{Laff} we obtain the following:
 $$
\frac{\delta L}{\delta x^k}= \dots-\frac{d}{dt}\Bigl(\sum_{j=1}^3\frac{\p\g_j}{\p u^k}a^j+\dots\Bigr)
+\frac{d^2}{dt^2}\bigl(\g_k+\dots\bigr).
 $$
Thus the symbol, i.e.\ coefficients of the highest derivatives $\dddot{x}^j$, is the skew-symmetric $3\times3$ matrix $A$
with entries $a_{ij}=\frac{\p\g_i}{\p u^j}-\frac{\p\g_j}{\p u^i}$, and so its determinant vanishes $\det A=0$.
This implies that the 3 equations in the system $\op{EL}(L)=0$ cannot be resolved wrt $\{\dddot{x}^j\}_{j=1}^3$
and hence this system is not equivalent to \eqref{BE0}.
 \end{proof}

Note that variationality of the constrained problem \eqref{BE3} can be approached via the Lagrangian including a
Lagrange multiplier 
 $$
L= \sum_{k=1}^3 \g_k(x,u)a^k+\g_0(x,u)+\lambda\cdot(|u|^2-1).
 $$
Variation by $\lambda$ gives the constraint $|u|^2=1$, and from this and its differential corollaries 
we can express the components $u^3$ and $a^3$. In what follows we do not specify constraints in the Lagrangian,
but since we deal with the unparametrized problem, we are free to impose such constraints whenever convenient.
As we noted in the introduction under such constraint equation \eqref{BE0} becomes \eqref{BE3} and there is one
linear dependency among these three equations; equivalently \eqref{BE1} holds identically, leaving two equations \eqref{BE2}.

\section{The flat and homogeneous cases}\label{S3}

In the flat case $g=ds^2_{\text{Eucl}}$ (due to conformal invariance also corresponding to spaces of positive 
and negative curvature) the equation has the following form in $\R^3$ 
(note that in this case $u=\dot{x}$, $a=\ddot{x}$):
 $$
\nabla_ua=-|a|^2u,\quad |u|^2=1,\quad g(u,a)=0.
 $$
Its solutions are circles, so conformal geodesics are sometimes also called conformal circles (CS).

The following 3rd order Lagrangian was proposed in \cite{BF1,BF2} (this reference contains more general Lagrangians, 
whose variations give helices and loxadromas, and conformal circles as special cases):
 $$
L=\frac{\ell\cdot V}{A^2}, 
 $$
where $\ell=|u|=\sqrt{g(u,u)}=\sqrt{\dot{x}_1^2+\dot{x}_2^2+\dot{x}_3^2}$ is the length,
 $$
A=\sqrt{\det\begin{pmatrix}g(u,u) & g(u,a)\\ g(u,a) & g(a,a)\end{pmatrix}}=|u\times a|
 $$
is the area of the $ua$ parallelogram and 
 $$
V= \det\begin{pmatrix}u_1 & u_2 & u_3\\ a_1 & a_2 & a_3\\ b_1 & b_2 & b_3\end{pmatrix}=|(u\times a)\cdot b|
 $$
is the volume of the parallelepiped formed by $u=\dot{x}$, $a=\nabla_uu=\ddot{x}$, $b=\nabla_ua=\dddot{x}$.
(Remark also that $V$ is the square root of the Gram matrix of the $uab$ tetrad, up to sign or orientation.)

The meaning of $L$ in the arc length parameter is the classical torsion $\tau$ of space curves 
(more general Lagrangians $L=L(\kappa,\tau)$ depending on curvature and torsion were considered in \cite{FGJL}). 
The Lagrangian $L$ is invariant with respect to conformal rescaling (the same for all metrics 
$g=\lambda\cdot ds^2_{\text{Eucl}}$) but is not covariant (it depends on Euclidean coordinates)
and hence is not conformally invariant. Also note that the action 
 \begin{equation}\label{ActL}
\mathcal{L}(\gamma)=\int_\gamma L\,dt
 \end{equation}
is reparametrization invariant, so its extremals are unparametrized curves $\gamma\subset\R^3$.

 \begin{theorem}\label{Thm2}
Extremals of $\mathcal{L}$ are precisely the conformal circles.
 \end{theorem}

This theorem is contained in \cite{BF1} (see the table on p.103529-5). Let us present another indirect approach,
which we will afterwards generalize to the curved context.

 \begin{proof}
The proof is divided in several steps (details are omitted but can be read off the next section).

1. Since the Lagrangian is linear in 3rd jet $b=\dddot{x}$, the Euler--Lagrange equation $\op{EL}=\frac{\delta L}{\delta x}$
is of order $<6$. Its derivatives by 5th jets are easily seen to be zero, so it is of order at most 4.

2. In fact, a direct computation shows that derivatives by 4th jets vanish as well, so the Euler--Lagrange equation $\op{EL}$
has order at most 3. 

3. The order is, indeed, equal to 3, because derivatives by 3rd jets are non-zero, 
and in fact the corresponding $3\times 3$ matrix $[\p\op{EL}_i/\p b_j]$ has rank 2.

4. The fact that this rank is non-maximal follows from the relation (when using summation we convert to
the standard convention of upper indices for vectors, etc)
 $$
\sum_1^3\dot{x}^i\frac{\delta L}{\delta x^i}=0.
 $$
Hence that rank cannot exceed 2, and it indeed equals 2.

5. The action and hence  the extremals are reparametrization invariant. Hence we can choose any parametrization 
to check that the equation is equivalent to the equation of conformal circles. 

6. We choose the parametrization as in the beginning of this section. Substitution of those equations in the 
Euler--Lagrange equation gives zeros, therefore $\op{EL}$ is a corollary of the CS equation.

7. On the other hand, the unparametrized CS equations consist of 2 equations, 
and the rank of $\op{EL}$ by 3rd jets is 2, hence we conclude the equivalence.
 \end{proof}

By Proposition \ref{P1} we know that the order of the Lagrangian can be reduced to 2. And indeed we can take
 \begin{equation}\label{2ndordL}
L'=L-\frac{d}{dt}\Bigl(\arctan\frac{\ddot{x}_1|u|^2-g(u,a)\dot{x}_1}
{(\dot{x}_2\ddot{x}_3-\dot{x}_3\ddot{x}_2)|u|}\Bigr)
=\frac{\dot{x}_1(\dot{x}_2\ddot{x}_3-\dot{x}_3\ddot{x}_2)}{(\dot{x}_2^2+\dot{x}_3^2)
\sqrt{\dot{x}_1^2+\dot{x}_2^2+\dot{x}_3^2}}.
 \end{equation}
We can take a more symmetric Lagrangian by averaging over indices (123) with the same extremals CS:
 $$
\tilde{L}=
\frac{\dot{x}_2\dot{x}_3(\dot{x}_2^4-\dot{x}_3^4)\ddot{x}_1+\dot{x}_3\dot{x}_1(\dot{x}_3^4-\dot{x}_1^4)\ddot{x}_2
+\dot{x}_1\dot{x}_2(\dot{x}_1^4-\dot{x}_2^4)\ddot{x}_3}{(\dot{x}_1^2+\dot{x}_2^2)(\dot{x}_2^2+\dot{x}_3^2)
(\dot{x}_3^2+\dot{x}_1^2)\sqrt{\dot{x}_1^2+\dot{x}_2^2+\dot{x}_3^2}} 
 $$

 \begin{rk}
Before coming to the general computation we experimented with homogeneous 3D conformal structures,
and observed that for a modified Lagrangian $L$ the equation of motion $\op{EL}(L)$ coincides with
the equation for conformal geodesics. 
We omit details of those computations, as 
they were superseded by our general computation in the next section.
 \end{rk}

\section{The general Lagrangian}\label{S4}

Having inspiration in the homogeneous cases let us now approach general curved conformal structures $[g]$, 
with the following Lagrangian
 \begin{equation}\label{LL}
L=\frac{\ell\cdot V}{A^2}, 
 \end{equation}
where $\ell=|u|=\sqrt{g(u,u)}$, $g(u,u)=g_{ij}\dot{x}^i\dot{x}^j$, is the length,
 $$
A=\sqrt{\det\begin{pmatrix}g(u,u) & g(u,a)\\ g(u,a) & g(a,a)\end{pmatrix}}
 $$
is the area, computed as before, and the curved version of the volume is
 $$
V= \sqrt{\det[g_{ab}]}\epsilon_{ijk}u^ia^jb^k,
 $$
where $a=\nabla_uu$, $b=\nabla_ua$. More precisely
 \begin{gather*}
u^k=\dot{x}^k,\qquad a^k=\ddot{x}^k+\Gamma^k_{ij}\dot{x}^i\dot{x}^j,\\
b^k=\dddot{x}^k+3\Gamma^k_{ij}\dot{x}^i\ddot{x}^j+S^k_{ijl}\dot{x}^i\dot{x}^j\dot{x}^l, 
 \end{gather*}
where $\Gamma^k_{ij}=\tfrac12g^{kl}(\p_ig_{jl}+\p_jg_{il}-\p_lg_{ij})$ are the Christoffel symbols
and $S^k_{ijl}=\p_l\Gamma^k_{ij}+\Gamma^k_{lp}\Gamma^p_{ij}$. Note that $S^k_{i[jl]}=\tfrac12R^k_{lij}$
is the Riemann curvature tensor. Note also the symmetry $S^k_{ijl}=S^k_{(ji)l}$, which 
can be changed to complete symmetry by $(ijl)$ due to appearance of $S^k_{ijl}$ in the formula for $b^k$.

The geometric meaning of Lagrangian \eqref{LL} is the following:
 \begin{lemma}
In an arbitrary parameter $t$ of the curve, the torsion is given by the formula $\tau=V/A^2$.
Therefore $L\,dt=\tau\,ds$ is parameter independent.
 \end{lemma}
 
 \begin{proof}
There exists a unique orthonormal frame $(T,N,B)$ along a nondegenerate curve $\gamma\subset M$
(a generic curve is such and this suffices for our purposes), where $T=\dot{x}/|\dot{x}|$ and the rest can be found 
from the generalized Frenet-Serret formulae (see e.g.\cite{Gu}):
 $$
\nabla_TT=\kappa N,\ \nabla_TN=-\kappa T+\tau B,\ \nabla_TB=-\tau T,
 $$
where $\kappa$ is the curvature and $\tau$ is the torsion. 
 
One can easily check that $V/A^2$ does not change under a reparametrization of the curve, so
it is sufficient to calculate it in the arc length parameter $s$, where $u=T$. Then $a=\kappa N$ and hence
 $$
A^2=\kappa^2.
 $$
Since $b=\nabla_ua=\nabla_u(\kappa N)=-\kappa^2 T+\kappa'N+\kappa\tau B$ (derivative wrt $s$) we conclude
 $$
V=\op{Vol}_g(u,a,b)= \kappa^2\tau.
 $$
The above two formulas imply the claim.
 \end{proof}
 
 \begin{rk}
As a by-product of the proof we note that along conformal geodesic \eqref{BE0} we have:
 $$
\kappa\tau=g(b,B)=g(P^\sharp u,B)=P(T,B)=\op{Ric}(T,B). 
 $$ 
This gives another approach to the Lagrangian of conformal geodesics via Ricci curvature.
 \end{rk}

Thus, we conclude that $L$ is covariant (coordinate-independent due to use of covariant derivative)
and isometry (but not conformally) invariant, and that the corresponding action 
$\mathcal{L}$ is projectively invariant. Thus extremals are reparametrization invariant curves. 

Let us note that in pseudo-Riemannian case the Lagrangian $L$ is not defined everywhere, as the denominator of \eqref{LL} 
may vanish. In this sense it is similar to the Kropina-like Lagrangian considered in \cite{CMMM}.
However $\kappa$ and $\tau$ are well-defined for nondegenerate curves (specified by the requirement that both
$T$ and $N$ in the Frenet-Serret  frame are non-null).

\section{The proof}\label{S5}

Theorem \ref{Thm} is a curved version of Theorem \ref{Thm2}; we will prove it
in steps following the strategy from Section~\ref{S3}. The computation is demanding; we use several tricks to simplify it. 
We first check that the Euler--Lagrange equation has order at most 5, 
then subsequetly reduce this to 4 and 3. Finally we show it has the same rank in $\dddot{x}$
as the equation of conformal geodesics and vanishes modulo this equation.

Recall that for $L=L(x,u,a,b)$ or equivalently $L=L(x,\dot{x},\ddot{x},\dddot{x})$ the variational derivative is
 \begin{equation}\label{ELord3}
\op{EL}_i=\frac{\delta L}{\delta x^i}=\frac{\p L}{\p x^i}-\frac{d}{dt}\frac{\p L}{\p\dot{x}^i}
+\frac{d^2}{dt^2}\frac{\p L}{\p\ddot{x}^i}-\frac{d^3}{dt^3}\frac{\p L}{\p\dddot{x}^i},
 \end{equation}
which obviously implies $\op{ord(EL)}<6$ because $L$ is affine in $\dddot{x}$.
 
 \begin{lemma}\label{ord5}
The Euler--Lagrange equation has $\op{ord(EL)}<5$.
 \end{lemma}
 
 \begin{proof}
Only two last two terms of $\op{EL}_i$ contribute affinely to 5th jets, and this contribution is
 $$
\Bigl(\frac{\p^2L}{\p\ddot{x}^i\dddot{x}^j}-\frac{\p^2L}{\p\ddot{x}^j\dddot{x}^i}\Bigr)\stackrel{\rm v}{x}{\!}^j.
 $$
Only the factor $V$ from $L$ contributes to the derivatives in $\dddot{x}$, but for this part 
the matrix $\Bigl[\frac{\p^2V}{\p\ddot{x}^i\p\dddot{x}^j}\Bigr]$ is skew-symmetric.
However, taking into account all factors we get symmetry and vanishing.

Denoting $\op{alt}_{[ij]}$ the alternation by $[ij]$, 
and writing also $\sqrt{|g|}=\sqrt{\det[g_{ab}]}$, $u_i=g_{ik}u^k$, etc, we get:
 \begin{align}
\quad\qquad
\frac{\p\op{EL}_i}{\p\!\stackrel{\rm v}{x}{\!}^j}=2\op{alt}_{[ij]}\frac{\p^2L}{\p\ddot{x}^i\dddot{x}^j}
=& \frac{\ell}{A^2}\Bigl(\frac{\p^2V}{\p\ddot{x}^i\p\dddot{x}^j}-\frac{\p^2V}{\p\ddot{x}^j\p\dddot{x}^i}\Bigr)-
\frac{\ell}{A^4}\Bigl(\frac{\p A^2}{\p\ddot{x}^i}\frac{\p V}{\p\dddot{x}^j}-\frac{\p A^2}{\p\ddot{x}^j}\frac{\p V}{\p\dddot{x}^i}
\Bigr)\notag\\
=&\,2\op{alt}_{[ij]}\Bigl(\frac{\ell}{A^2}\sqrt{|g|}\epsilon_{kij}u^k
-\frac{2\ell}{A^4}\Bigl(\ell^2a_i-(u\cdot a)u_i\Bigr)\sqrt{|g|}\epsilon_{pqj}u^pa^q\Bigr).\quad\qquad
\label{eq_01}
 \end{align}
Up to a common nonzero factor this expression is
 $$
\epsilon_{ikj}u^kA^2- \Bigl(\ell^2a_j-(u\cdot a)u_j\Bigr)\epsilon_{ipq}u^pa^q
+\Bigl(\ell^2a_i-(u\cdot a)u_i\Bigr)\epsilon_{jpq}u^pa^q=0.
 $$
The last equality is obvious if $i=j$. In the case $i\neq j$ one can permute indices to make $i=1$, $j=2$,
and verify it directly. To simplify computations one may assume (due to covariance and reparametrizaton invariance)
that $\ell^2=1$, $u\cdot a=0$ and $g_{ij}=\delta_{ij}$ at the given point. Then the result is straightforward.
 \end{proof}

Denote by $O(d)$ a differential polynomial on $d$-jets. Then for $L=L(x,\dot{x},\ddot{x},\dddot{x})$
we get the following formulae:
 \begin{align*}
\frac{\p L}{\p x^i} &=O(3),\\
\frac{d}{dt}\frac{\p L}{\p\dot{x}^i} &= \frac{\p^2L}{\p\dot{x}^i\p\dddot{x}^j}\stackrel{\rm iv}{x}{\!}^j+O(3),\\ 
\frac{d^2}{dt^2}\frac{\p L}{\p\ddot{x}^i} &= \frac{\p^2L}{\p\ddot{x}^i\p\dddot{x}^j}\stackrel{\rm v}{x}{\!}^j
+\Bigl(2\frac{\p^3L}{\p x^k\p\ddot{x}^i\p\dddot{x}^j}\dot{x}^k
+2\frac{\p^3L}{\p\dot{x}^k\p\ddot{x}^i\p\dddot{x}^j}\ddot{x}^k
+2\frac{\p^3L}{\p\ddot{x}^k\p\ddot{x}^i\p\dddot{x}^j}\dddot{x}^k
+\frac{\p^2L}{\p\ddot{x}^i\p\ddot{x}^j}\Bigr)\stackrel{\rm iv}{x}{\!}^j +O(3),\\ 
\frac{d^3}{dt^3}\frac{\p L}{\p\dddot{x}^i} &= \frac{\p^2L}{\p\ddot{x}^j\p\dddot{x}^i}\stackrel{\rm v}{x}{\!}^j
+\Bigl(3\frac{\p^3L}{\p x^k\p\ddot{x}^j\p\dddot{x}^i}\dot{x}^k
+3\frac{\p^3L}{\p\dot{x}^k\p\ddot{x}^j\p\dddot{x}^i}\ddot{x}^k
+3\frac{\p^3L}{\p\ddot{x}^k\p\ddot{x}^j\p\dddot{x}^i}\dddot{x}^k
+\frac{\p^2L}{\p\dot{x}^j\p\dddot{x}^i}\Bigr)\stackrel{\rm iv}{x}{\!}^j +O(3).
 \end{align*}
 
Now we can compute the coefficients of the 4th jets. 
 
 \begin{lemma}\label{ord4}
The Euler--Lagrange equation has $\op{ord(EL)}<4$.
 \end{lemma}
 
 \begin{proof} 
Fix an arbitrary point $x_0\in M$, where we are going to verify the claim.
We will use the following notation: $f\doteq h$ iff $f(x_0)=h(x_0)$. 
Due to covariance we may fix a normal coordinate system centered at the point $x_0$,
so we have: $g_{ij}\doteq\delta_{ij}$, $\sqrt{|g|}\doteq1$ and
$\p_k(g_{ij})\doteq0$, $\p_k(\sqrt{|g|})\doteq0$, $\Gamma^{k}_{ij}\doteq0$. 

By reparametrization invariance of the Lagrangian, we may also assume $\ell=1$. This implies 
$g(u,a)=0$, $g(u,b)=-|a|^2$, $A^2=|a|^2$ and $V^2=|a|^2|b|^2-|a|^6-g(a,b)^2$.

In order to keep the use of index notation consistent, we use the notation $\Gamma_{ijk} := g_{ks} \Gamma^{s}_{ij}$. 
Now we subsequently calculate various coefficients of the fourth order terms of $\op{EL}$.

1. We start with
 \begin{equation}\label{eq_10}
\frac{\p^2L}{\p\dot{x}^i\p\dddot{x}^j}= \frac{\ell\,\p_{\dot{x}^i}\p_{\dddot{x}^j}(V)}{A^2} +
\frac{\p_{\dot{x}^i}(\ell)\,\p_{\dddot{x}^j}(V)}{A^2} - \frac{\ell\,\p_{\dot{x}^i}(A^2)\,\p_{\dddot{x}^j}(V)}{A^4}.
 \end{equation}

We compute:
 \begin{equation}\label{eq_11}
\p_{\dot{x}^i}(\ell^2) = 2u_i,\quad \p_{\dot{x}^i}(\ell) = \frac{u_i}{\ell}\doteq u_i.
 \end{equation}
and 
 $$
\p_{\dot{x}^i}\p_{\dddot{x}^j}(V) = \sqrt{|g|}\epsilon_{pqj}\p_{\dot{x}^i}(u^p a^q)
= \sqrt{|g|}\epsilon_{iqj} a^q + 2\sqrt{|g|}\epsilon_{pqj} u^p \Gamma^q_{ci} u^c \doteq\epsilon_{iqj} a^q
 $$
where we use
 \begin{equation}\label{eq_11.5}
\p_{\dot{x}^i}(a^k) = 2\Gamma_{ij}^ku^j\doteq0.
 \end{equation}
The latter formula implies
 \begin{equation}\label{eq_12}
\p_{\dot{x}^i}(|a|^2) = 4 \Gamma_{ij}^k u^ja_k\doteq0,\quad
\p_{\dot{x}^i}(g(u,a))  =  g_{kl} \p_{\dot{x}^i}(u^ka^l) = a_i + 2u_k\Gamma_{ij}^ku^j\doteq a_i.
 \end{equation}
Hence we get
 \begin{equation}\label{eq_13}
\p_{\dot{x}^i}(A^2) = \p_{\dot{x}^i}(\ell^2 |a|^2 - g(u,a)^2) 
= 4 \ell^2 \Gamma_{il}^k u^l a_k + 2 |a|^2 u_i - 2 g(u,a) (a_i + 2 \Gamma_{il}^k u^l u_k)\doteq 2|a|^2u_i. 
 \end{equation}
Using these formulae we conclude that, at $x_0$, the required coefficient \eqref{eq_10} is
 \begin{equation}\label{eq_14}
\frac{\p^2L}{\p\dot{x}^i\p\dddot{x}^j}\doteq \frac1{|a|^2}\Bigl(\epsilon_{iqj}a^q-\epsilon_{pqj} u_iu^pa^q\Bigr).
 \end{equation}

2. Using formula \eqref{eq_01} from Lemma \ref{ord5} we write
 \begin{equation}\label{eq_20}
\frac{\p^3L}{\p\ddot{x}^k\p\ddot{x}^i\p\dddot{x}^j}= 
\p_{\ddot{x}^k}\Bigl(\frac{\ell}{A^2}\sqrt{|g|}\epsilon_{lij} u^l - \frac{2\ell}{A^4}(\ell^2a_i - (u\cdot a)u_i)\,\sqrt{|g|} \epsilon_{pqj}u^pa^q\Bigr)
 \end{equation}
and using $\p_{\ddot{x}^k}(A^2)\doteq 2a_k$ we obtain
 \begin{equation*}
\p_{\ddot{x}^k}\Bigl(\frac{\ell}{A^2}\sqrt{|g|}\epsilon_{lij} u^l\Bigr) =
-\frac{\ell}{A^4}\sqrt{|g|}\epsilon_{lij}u^l\p_{\ddot{x}^k}(A^2)\doteq\frac{-2 }{|a|^4} \epsilon_{lij} u^l a_k.
 \end{equation*}

Then we compute
 \begin{equation*}
    \begin{split}
\p_{\ddot{x}^k}\Bigl( -\frac{2\ell}{A^4}(\ell^2a_i &- g(u,a)u_i)\,\sqrt{|g|}\epsilon_{pqj}u^p a^q\Bigr) 
 \doteq  -2\epsilon_{pqj}u^p \left[ \p_{\ddot{x}^k}\Bigl(\frac{a_ia^q}{A^4}\Bigr)  
- u_i\p_{\ddot{x}^k}\Bigl(\frac{g(u,a)a^q}{A^4}\Bigr) \right] \\
& \doteq -2\epsilon_{pqj}u^p\left[\frac{\p_{\ddot{x}^k}(a_i a^q)}{|a|^4} - \frac{2a_ia^q\p_{\ddot{x}^k}(A^2)}{|a|^6} - 
\frac{u_i\,\p_{\ddot{x}^k}(g(u,a)a^q )}{|a|^4} \right]\\
& \doteq -2\epsilon_{pqj}u^p\left[\frac1{|a|^4}(a_i\delta^q_k+\delta_{ik}a^q) -\frac4{|a|^6}a_ia_ka^q 
-\frac1{|a|^4}u_iu_ka^q\right].
    \end{split}
 \end{equation*}
At $x_0$, we have 
 $$
\frac{\p^3L}{\p\ddot{x}^k\p\ddot{x}^i\p\dddot{x}^j} \doteq -\frac2{|a|^4}\epsilon_{lij}u^la_k 
- \frac2{|a|^4}\epsilon_{pkj}u^pa^i - \frac2{|a|^4}\epsilon_{pqj}u^p\delta_{ik}a^q + \frac2{|a|^4}\epsilon_{pqj}u^pa^qu_iu_k
+ \frac8{|a|^6}\epsilon_{pqj}u^pa^qa_ia_k 
 $$
and this iterates into
 \begin{align}\label{eq_21}
2\frac{\p^3L}{\p\ddot{x}^k\p\ddot{x}^i\p\dddot{x}^j} &- 3\frac{\p^3L}{\p\ddot{x}^k\p\ddot{x}^j\p\dddot{x}^i} \doteq
- \frac{10}{|a|^4}\epsilon_{lij}u^la_k - \frac4{|a|^4}\epsilon_{pkj}u^pa_i - \frac4{|a|^4}\epsilon_{pqj}u^p\delta_{ik}a^q 
+ \frac4{|a|^4}\epsilon_{pqj}u_iu_ku^pa^q \\
& +\frac{16}{|a|^6}\epsilon_{pqj}a_iu^pa^qa_k +\frac6{|a|^4}\epsilon_{pki}u^pa_j +\frac6{|a|^4}\epsilon_{pqi}u^p\delta_{jk}a^q  - \frac6{|a|^4}\epsilon_{pqi}u_ju_ku^pa^q - \frac{24}{|a|^6}\epsilon_{pqi}a_ju^pa^qa_k.\notag
 \end{align}

3. Next we consider
 $$
\frac{\p^2L}{\p\ddot{x}^i\p\ddot{x}^j} = 
\ell\left[\frac{\p_{\ddot{x}^i}\p_{\ddot{x}^j}(V)}{A^2} - \frac{\p_{\ddot{x}^i}(A^2)\p_{\ddot{x}^j}(V)}{A^4} 
- \frac{\p_{\ddot{x}^i}(V)\p_{\ddot{x}^j}(A^2)}{A^4}  - \frac{V\p_{\ddot{x}^i}\p_{\ddot{x}^j}(A^2)}{A^4} 
+ \frac{2V\p_{\ddot{x}^i}(A^2)\,\p_{\ddot{x}^j}(A^2)}{A^6} \right].
 $$
We compute
 $$
\p_{\ddot{x}^j}(V) = \sqrt{|g|}\epsilon_{pqr}u^p\p_{\ddot{x}^j}(a^q b^r) 
= \sqrt{|g|}\epsilon_{pjr}u^pb^r + 3\sqrt{|g|}\epsilon_{pqr}u^pa^q\Gamma_{cj}^r u^c,
 $$
then
 $$
\p_{\ddot{x}^i}\p_{\ddot{x}^j}(V) = \sqrt{|g|}\epsilon_{pjr}u^p\p_{\ddot{x}^i}(b^r) + 3\sqrt{|g|}\epsilon_{pqr}u^p\delta^q_i \Gamma_{cj}^r u^c 
 = 3\sqrt{|g|}\epsilon_{pjr}u^p\Gamma_{ic}^ru^c + 3\sqrt{|g|}\epsilon_{pir}u^p\Gamma_{cj}^ru^c \doteq 0,
 $$
and also
 $$
\p_{\ddot{x}^i}\p_{\ddot{x}^j}(A^2) = 2\p_{\ddot{x}^i}(\ell^2 a_j - g(u,a) u_j) =  2(\delta_{ij} - u_iu_j).
 $$
Thus, at $x_0$, we have
\begin{equation}\label{eq_31}
    \begin{split}
\frac{\p^2L}{\p\ddot{x}^i\p\ddot{x}^j}
& \doteq  -\frac2{|a|^4}\epsilon_{pjr}u^pb^ra_i  -\frac2{|a|^4}\epsilon_{pir}u^pb^ra_j -\frac{2V}{|a|^4}(\delta_{ij} - u_i u_j) 
+ \frac{8V}{|a|^6}a_ia_j \\ 
& \doteq -\left(\frac2{|a|^4}u^p(\epsilon_{pjr}a_i+\epsilon_{pir}a_j) +\frac2{|a|^4}\epsilon_{pqr}u^pa^q (\delta_{ij} - u_i u_j) 
- \frac8{|a|^6}\epsilon_{pqr}u^pa^qa_ia_j\right) b^r
    \end{split}
\end{equation}

4. Using again formula \eqref{eq_01} we find that
\begin{equation}\label{eq_40}
    \begin{split}
\frac{\p^3L}{\p\dot{x}^k\p\ddot{x}^i\p\dddot{x}^j} & = 
\p_{\dot{x}^k}\left(\frac{\ell}{A^2}\sqrt{|g|}\epsilon_{lij}u^l - \frac{2 \ell }{A^4} (\ell^2 a_i - g(u,a) u_i)\,\sqrt{|g|}\epsilon_{pqj} u^p a^q\right) \\
& = \sqrt{|g|}\epsilon_{lij}\p_{\dot{x}^k}\Bigl(\frac{\ell u^l}{A^2}\Bigr) - 
2\,\sqrt{|g|}\epsilon_{pqj}\left(\p_{\dot{x}^k}\Bigl(\frac{\ell^3}{A^4}a_iu^pa^q\Bigr)-
\p_{\dot{x}^k}\Bigl(\frac{\ell}{A^4}g(u,a)u_iu^pa^q\Bigr) \right).
    \end{split}
\end{equation}

Using \eqref{eq_11} and \eqref{eq_13} we compute first
\begin{equation}
    \begin{split}
\p_{\dot{x}^k}\Bigl(\frac{\ell u^l}{A^2}\Bigr) & = \delta^l_k\frac{\ell}{A^2} + u^l\frac{\p_{\dot{x}^k}(\ell)}{A^2} 
- \ell u^l\frac{\p_{\dot{x}^k}(A^2)}{A^4} \\
& \doteq \delta^l_k\frac1{|a|^2} + \frac{u^lu_k}{|a|^2} - \frac2{|a|^2}u^lu_k = \frac{\delta^l_k}{|a|^2} - \frac{u^lu_k}{|a|^2},
    \end{split}
\end{equation}
Then we compute using \eqref{eq_11.5} and \eqref{eq_13}
 \begin{align}
2\sqrt{|g|}\epsilon_{pqj}\p_{\dot{x}^k}\left(\frac{\ell^3}{A^4}a_i u^p a^q\right) &= 
2\sqrt{|g|}\epsilon_{pqj}\left(\frac{\ell^3}{A^4}\bigl(\delta^p_ka_ia^q + u^p\p_{\dot{x}^k}(a_i a^q)\bigr) + 
\frac{3\ell^2}{A^4}a_iu^pa^q\p_{\dot{x}^k}(\ell) - \frac{2\ell^3}{A^6}a_iu^pa^q\p_{\dot{x}^k}(A^2) \right)\notag\\
&\doteq \frac2{|a|^2}\epsilon_{kqj}a_ia^q -  \frac2{|a|^4}\epsilon_{pqj}a_iu^pa^qu_k. \label{eq_41}
 \end{align}
Then we compute using \eqref{eq_12}
 \begin{equation}\label{eq_42}
2\sqrt{|g|}\epsilon_{pqj}\p_{\dot{x}^k}\left(\frac{\ell}{A^4}g(u,a)u_iu^pa^q\right)  =  
2\sqrt{|g|}\epsilon_{pqj}\frac1{A^4}\p_{\dot{x}^k}(g(u,a)u_iu^pa^q)
\doteq\frac2{|a|^4}\epsilon_{pqj}u^pu_ia^qa_k.
 \end{equation}
Thus, at $x_0$, we have
 $$
\frac{\p^3L}{\p\dot{x}^k\p\ddot{x}^i\p\dddot{x}^j} \doteq \frac{\epsilon_{kij}}{|a|^2} - \epsilon_{lij}\frac{u^lu_k}{|a|^2} 
 - \frac2{|a|^2}\epsilon_{kqj}a_ia^q + \frac2{|a|^4}\epsilon_{pqj}u^pa^qa_iu_k + \frac2{|a|^4}\epsilon_{pqj}u^pa^qu_ia_k,
 $$
which implies
\begin{align}
2\frac{\p^3L}{\p\dot{x}^k\p\ddot{x}^i\p\dddot{x}^j} - 3\frac{\p^3L}{\p\dot{x}^k\p\ddot{x}^j\p\dddot{x}^i} 
& \doteq \frac{5\epsilon_{kij} }{|a|^2} - 5\epsilon_{lij}\frac{u^lu_k}{|a|^2}  - \frac4{|a|^2}\epsilon_{kqj}a_ia^q + 
\frac4{|a|^4}\epsilon_{pqj}u^pa^qa_iu_k + \frac4{|a|^4}\epsilon_{pqj}u^pa^qu_ia_k\notag\\ 
& + \frac{6}{|a|^2}\epsilon_{kqi} a_j a^q - \frac{6}{|a|^4}\epsilon_{pqi} u^p a^q a_j u_k 
- \frac{6}{|a|^4} \epsilon_{pqi} u^p a^q u_j a_k. \label{eq_43}
\end{align}

5. At last, again using \eqref{eq_01}, let us compute
 \begin{equation}\label{eq_50}
\frac{\p^3L}{\p x^k\p\ddot{x}^i\p\dddot{x}^j} = \p_{x^k}\left(\frac{\ell}{A^2}\sqrt{|g|}\epsilon_{lij}u^l 
- \frac{2\ell}{A^4}(\ell^2a_i - (u\cdot a)u_i)\,\sqrt{|g|}\epsilon_{pqj}u^pa^q\right).
 \end{equation}

We have
 \begin{equation}\label{eq_51}
\p_{x^k}(\ell^2)=\p_k(g_{cd})u^cu^d,\qquad \p_{x^k}(\ell)=\frac1{2\ell}\p_k(g_{cd})u^cu^d\doteq0,
 \end{equation}
and
 \begin{equation}\label{eq_52}
\p_{x^k}(a^q)=\p_k(\Gamma^q_{cd})u^cu^d,\qquad
\p_{x^k}(|a|^2)\doteq 2a_q\p_k(\Gamma^q_{cd})u^cu^d,
 \end{equation}
and also
 \begin{equation*}
\p_{x^k}(g(u,a) )\doteq u_q\p_k(\Gamma^q_{cd})u^cu^d.
 \end{equation*}
Now the product rule and equations \eqref{eq_51}-\eqref{eq_52} yield
 \begin{equation}\label{eq_53}
\p_{x^k}(A^2) = \p_{x^k}\bigl(\ell^2 |a|^2 - g(u,a)^2\bigr) 
\doteq \p_{x^k}(\ell^2)\,|a|^2 + \p_{x^k}(|a|^2) \doteq 2a_q\p_k(\Gamma^q_{cd})u^cu^d.
 \end{equation}

Consequently, using \eqref{eq_51} and \eqref{eq_53}, we obtain
 \begin{equation}
\p_{x^k}\left(\frac{\ell}{A^2}\sqrt{|g|}\right) \doteq - \frac2{|a|^4}a_q\p_k(\Gamma^q_{cd})u^cu^d  
 \end{equation}

Next we compute, again using \eqref{eq_51} and \eqref{eq_53},
 \begin{equation}\label{eq_54}
    \begin{split}
\p_{x^k}\left(\frac{2 \ell^3}{A^4}\sqrt{|g|}a_iu^pa^q\right) & \doteq 
2u^p\left[\p_{x^k}\Bigl(\frac{\ell^3}{A^4}\Bigr)a_i a^q + \frac1{|a|^4}\p_{x^k}(\sqrt{|g|}a_ia^q)\right] \\
& \doteq - \frac8{|a|^6}u^pa^qa_ia_s\p_k(\Gamma^s_{cd})u^cu^d 
+ \frac2{|a|^4}\bigl(\p_k(\Gamma_{cdi})u^cu^du^pa^q + \p_k(\Gamma_{cd}^q)u^cu^du^pa_i\bigr).
    \end{split}
\end{equation}
The remaining term to compute is
 \begin{equation}\label{eq_55}
\p_{x^k}(2u^p\frac{\ell}{A^4}\sqrt{|g|}\,g(u,a)u_ia^q) \doteq 
\frac2{|a|^4}u^p\p_k\bigl(\sqrt{|g|}\,g(u,a)u_ia^q)\bigr) 
\doteq \frac2{|a|^4}u^pu_ia^qu_s\p_k(\Gamma^s_{cd})u^cu^d.
 \end{equation}
 
				\pagebreak

Now substituting \eqref{eq_53}-\eqref{eq_55} into \eqref{eq_50} we get
\begin{equation*}
    \begin{split}
\frac{\p^3L}{\p x^k\p\ddot{x}^i\p\dddot{x}^j} \doteq  & 
- \frac2{|a|^4}\epsilon_{lij}u^la_s\p_k(\Gamma^s_{cd})u^cu^d 
+ \frac8{|a|^6}\epsilon_{pqj}u^pa_ia^qa_s\p_k(\Gamma^s_{cd})u^cu^d  \\
&  - \frac2{|a|^4}\epsilon_{pqj}u^pa^q\p_k(\Gamma_{cdi})u^cu^d 
- \frac2{|a|^4}\epsilon_{pqj}u^p\p_k(\Gamma_{cd}^q)u^cu^da_i 
+ \frac2{|a|^4}\epsilon_{pqj}u_iu^pa^qu_s\p_k(\Gamma^s_{cd})u^cu^d.
    \end{split}
\end{equation*}
From this we deduce (using $\epsilon_{lji} = -\epsilon_{lij}$ to simplify)
\begin{equation}\label{eq_56}
    \begin{split}
\!\!2\frac{\p^3L}{\p x^k\p\ddot{x}^i\p\dddot{x}^j} &- 3\frac{\p^3L}{\p x^k\p\ddot{x}^j\p\dddot{x}^i} \doteq \\ 
&\!\!\! - \frac{10}{|a|^4}\epsilon_{lij} u^l a_s \p_k(\Gamma^s_{cd}) u^c u^d 
+ \frac{16}{|a|^6}\epsilon_{pqj} u^p  a_i a^q  a_s \p_k(\Gamma^s_{cd}) u^c u^d  
- \frac{4}{|a|^4}\epsilon_{pqj} u^p a^q  \p_k(\Gamma_{cdi}) u^c u^d \\
&\!\!\! - \frac{4}{|a|^4}\epsilon_{pqj} u^p\p_k(\Gamma_{cd}^q) u^c u^d a_i 
+ \frac{4 }{|a|^4}\epsilon_{pqj} u_i u^p a^q u_s\p_k(\Gamma^s_{cd}) u^c u^d 
- \frac{24}{|a|^6}\epsilon_{pqi} u^p  a_j a^q  a_s \p_k(\Gamma^s_{cd}) u^c u^d \\ 
&\!\!\! + \frac{6}{|a|^4}\epsilon_{pqi} u^p a^q \p_k(\Gamma_{cdj}) u^c u^d 
+ \frac{6}{|a|^4}\epsilon_{pqi} u^p\p_k(\Gamma_{cd}^q) u^c u^d a_j 
- \frac{6}{|a|^4}\epsilon_{pqi} u_j u^p a^q u_s\p_k(\Gamma^s_{cd}) u^c u^d.
    \end{split}
\end{equation}

We are now ready to show that the fourth-order coefficients of the Euler--Lagrange equation vanish. 
We freely rename dummy indices in order to observe cancellations. We also exploit the relations
 $$
\dot{x}^k=u^k,\qquad \ddot{x}^k \doteq a^k,\qquad \dddot{x}^r \doteq b^r - \p_d(\Gamma_{bc}^r) u^b u^c u^d
 $$
From the formulae before Lemma \ref{ord5} we get:
\begin{equation}\label{eq_fourth_order_coeff}
    \begin{split}
\frac{\p EL_i}{\p\!\stackrel{\rm iv}{x}{\!}^j} & =
\frac{\p^2L}{\p\ddot{x}^i\p\ddot{x}^j} - \frac{\p^2L}{\p\dot{x}^i\p\dddot{x}^j} - \frac{\p^2L}{\p\dot{x}^j\p\dddot{x}^i} 
+ \Bigl(2\frac{\p^3L}{\p\ddot{x}^r\p\ddot{x}^i\p\dddot{x}^j} 
- 3\frac{\p^3L}{\p\ddot{x}^r\p\ddot{x}^j\p\dddot{x}^i}\Bigr)\dddot{x}^r  \\ 
& + \Bigl(2\frac{\p^3L}{\p\dot{x}^k\p\ddot{x}^i\p\dddot{x}^j} - 3\frac{\p^3L}{\p\dot{x}^k\p\ddot{x}^j\p\dddot{x}^i}\Bigr)\ddot{x}^k 
+ \Bigl(2\frac{\p^3L}{\p x^d\p\ddot{x}^i\p\dddot{x}^j} - 3\frac{\p^3L}{\p x^d\p\ddot{x}^j\p\dddot{x}^i}\Bigr)\dot{x}^d \\ 
& \doteq \frac{\p^2L}{\p\ddot{x}^i\p\ddot{x}^j} -\frac{\p^2L}{\p\dot{x}^i\p\dddot{x}^j} - \frac{\p^2L}{\p\dot{x}^j\p\dddot{x}^i} 
+ \Bigl(2\frac{\p^3L}{\p\ddot{x}^r\p\ddot{x}^i \p\dddot{x}^j} -3\frac{\p^3L}{\p\ddot{x}^r\p\ddot{x}^j\p\dddot{x}^i}\Bigr)b^r\\ 
& - \Bigl(2\frac{\p^3L}{\p\ddot{x}^r\p\ddot{x}^i\p\dddot{x}^j} 
- 3 \frac{\p^3L}{\p\ddot{x}^r \p\ddot{x}^j \p\dddot{x}^i}\Bigr) \p_d(\Gamma_{bc}^r) u^b u^c u^d  
+ \Bigl(2\frac{\p^3L}{\p\dot{x}^k \p\ddot{x}^i \p\dddot{x}^j} -3\frac{\p^3L}{\p\dot{x}^k \p\ddot{x}^j \p\dddot{x}^i}\Bigr)a^k\\ 
& +  \Bigl(2\frac{\p^3L}{\p x^d \p\ddot{x}^i \p\dddot{x}^j} - 3\frac{\p^3L}{\p x^d \p\ddot{x}^j \p\dddot{x}^i}\Bigr) u^d.
    \end{split}
\end{equation}

First, note that by \eqref{eq_43} we have:
\begin{equation}\label{4th1}
    \begin{split}
\left(2\frac{\p^3L}{\p\dot{x}^k \p\ddot{x}^i \p\dddot{x}^j}\right. 
&-\left. 3\frac{\p^3L}{\p\dot{x}^k \p\ddot{x}^j \p\dddot{x}^i}\right) a^k 
\doteq 
\frac{5\epsilon_{kij}}{|a|^2} a^k - 5\epsilon_{lij}\frac{u^l u_k a^k}{|a|^2} - \frac{4}{|a|^2}\epsilon_{kqj} a_i a^q a^k 
+ \frac{4}{|a|^4}\epsilon_{pqj} u^p a^q a_i u_k a^k \\ 
& + \frac{4}{|a|^4}\epsilon_{pqj} u^p a^q u_i a_k a^k + \frac{6}{|a|^2}\epsilon_{kqi} a_j a^q a^k 
- \frac{6}{|a|^4}\epsilon_{pqi} u^p a^q a_j u_k a^k - \frac{6}{|a|^4}\epsilon_{pqi} u^p a^q u_j a_k a^k  \\ 
& = \frac5{|a|^2}\epsilon_{kij} a^k + \frac4{|a|^2}\epsilon_{pqj} u^p a^q u_i - \frac6{|a|^2}\epsilon_{pqi} u^p a^q u_j, 
\end{split}
\end{equation}
where we used that $\epsilon_{kql}a^qa^k = 0$ by skew-symmetry.
Next, using \eqref{eq_21}-\eqref{eq_31} and $g(u,b)=-|a|^2$, we get:
\begin{equation}\label{4th2}
    \begin{split}
\frac{\p^2L}{\p\ddot{x}^i \p\ddot{x}^j} &+ \left(2\frac{\p^3L}{\p\ddot{x}^r \p\ddot{x}^i \p\dddot{x}^j} 
- 3\frac{\p^3L}{\p\ddot{x}^r \p\ddot{x}^j \p\dddot{x}^i}\right) b^r \\
& \doteq \left( -\frac{2}{|a|^4}\epsilon_{pjr} u^p a_i  - \frac{2}{|a|^4}\epsilon_{pir} u^p a_j 
- \frac{2}{|a|^4}\epsilon_{pqr} u^p a^q (\delta_{ij} - u_i u_j) + \frac{8}{|a|^6}\epsilon_{pqr} u^p a^q a_i a_j  \right. \\ 
& - \frac{10}{|a|^4}\epsilon_{lij} u^l a_r - \frac{4}{|a|^4}\epsilon_{prj} u^p a_i - \frac{4}{|a|^4}\epsilon_{pqj} u^p\delta_{ir} a^q  
+ \frac{4}{|a|^4} \epsilon_{pqj} u_i u_r u^p a^q + \frac{16}{|a|^6}\epsilon_{pqj} a_i u^p a^q a_r \\ 
& \left. + \frac{6 }{|a|^4}\epsilon_{pri} u^p a_j + \frac{6}{|a|^4}\epsilon_{pqi} u^p\delta_{jr} a^q  
- \frac{6}{|a|^4}\epsilon_{pqi} u_j u_r u^p a^q - \frac{24}{|a|^6}\epsilon_{pqi} a_j u^p a^q a_r \right) b^r \\ 
& = \frac{2}{|a|^4}\epsilon_{pjr} u^p a_i b^r - \frac{8}{|a|^4}\epsilon_{pir} u^p a_j b^r 
- \frac{2}{|a|^4}\epsilon_{pqr} u^p a^q b^r (\delta_{ij} - u_i u_j) + \frac{8}{|a|^6}\epsilon_{pqr} u^p a^q b^r a_i a_j  \\ 
&  -\frac{10}{|a|^4}\epsilon_{lij}u^lg(a,b) -\frac4{|a|^2}\epsilon_{pqj}u_iu^pa^q +\frac{16}{|a|^6}\epsilon_{pqj}a_iu^pa^qg(a,b) \\ 
& + \frac6{|a|^2}\epsilon_{pqi} u_j u^p a^q - \frac{24}{|a|^6}\epsilon_{pqi} a_j u^p a^q g(a,b) 
+ \frac{6}{|a|^4}\epsilon_{pqi} u^p a^q b_j - \frac{4}{|a|^4}\epsilon_{pqj} u^p a^q b_i.
    \end{split}
\end{equation}

We now compare and simplify two terms that appear in equation \eqref{eq_fourth_order_coeff} 
using \eqref{eq_21} and \eqref{eq_56}:
\begin{equation}\label{4th3}
    \begin{split}
\Bigl(2\frac{\p^3L}{\p x^d \p\ddot{x}^i \p\dddot{x}^j} &- 3\frac{\p^3L}{\p x^d \p\ddot{x}^j \p\dddot{x}^i}\Bigr ) u^d
- \Bigl(2\frac{\p^3L}{\p\ddot{x}^r \p\ddot{x}^i \p\dddot{x}^j} - 3\frac{\p^3L}{\p\ddot{x}^r \p\ddot{x}^j \p\dddot{x}^i}) \p_d(\Gamma_{bc}^r) u^b u^c u^d \\ 
&\hspace{-15pt} \doteq
\left(-\frac{10}{|a|^4}\epsilon_{lij} u^l a_r \p_d(\Gamma^r_{bc}) u^b u^c 
+ \frac{16}{|a|^6}\epsilon_{pqj} u^p a_i a^q a_r\p_d(\Gamma^r_{bc}) u^b u^c
- \frac4{|a|^4}\epsilon_{pqj} u^p a^q\p_d(\Gamma_{bci}) u^b u^c \right. \\
&\hspace{-15pt} - \frac4{|a|^4}\epsilon_{pqj} u^p\p_d(\Gamma_{bc}^q) u^b u^c a_i 
+ \frac4{|a|^4}\epsilon_{pqj} u_i u^p a^q u_r\p_d(\Gamma^r_{bc}) u^b u^c 
- \frac{24}{|a|^6}\epsilon_{pqi} u^p a_j a^q  a_r \p_d(\Gamma^r_{bc}) u^b u^c \\ 
&\hspace{-15pt} \left. + \frac6{|a|^4}\epsilon_{pqi} u^p a^q\p_d(\Gamma_{bcj}) u^b u^c 
+ \frac6{|a|^4}\epsilon_{pri} u^p  \p_d(\Gamma_{bc}^r) u^b u^c a_j 
- \frac6{|a|^4}\epsilon_{pqi} u_j u^p a^q u_r\p_d(\Gamma^r_{bc}) u^b u^c \right) u^d \\ 
&\hspace{-15pt} - \left(-\frac{10}{|a|^4}\epsilon_{lij} u^l a_r - \frac4{|a|^4}\epsilon_{prj} u^p a^i 
- \frac4{|a|^4}\epsilon_{pqj} u^p \delta_{ir} a^q + \frac4{|a|^4}\epsilon_{pqj} u_i u_r u^p a^q   
+ \frac{16}{|a|^6}\epsilon_{pqj} a_i u^p a^q a_r \right. \\ 
&\hspace{-15pt} \left. + \frac6{|a|^4}\epsilon_{pri} u^p a_j + \frac6{|a|^4}\epsilon_{pqi} u^p\delta_{jr} a^q  
- \frac6{|a|^4}\epsilon_{pqi}u_ju_ru^pa^q - \frac{24}{|a|^6}\epsilon_{pqi}a_ju^pa^qa_r\right)\p_d(\Gamma_{bc}^r) u^bu^cu^d\\ 
&\hspace{-15pt} = \frac6{|a|^4}\epsilon_{pqi} u^p a^q\p_d(\Gamma_{bcj}) u^b u^c u^d - \frac4{|a|^4}\epsilon_{pqj} u^p a^q  \p_d(\Gamma_{bci}) u^b u^c u^d - \frac6{|a|^4}\epsilon_{pqi} u^p a^q \p_d(\Gamma_{bcj}) u^b u^c u^d\\
&\hspace{-15pt} + \frac4{|a|^4}\epsilon_{pqj} u^p a^q\p_d(\Gamma_{bci}) u^b u^c u^d = 0. 
    \end{split}
\end{equation}

Finally, substituting \eqref{eq_14} and \eqref{4th1}-\eqref{4th3} into \eqref{eq_fourth_order_coeff} we
get the following expression for the fourth order coefficient of the Euler--Lagrange equation:
\begin{equation}
    \begin{split}
\frac{\p EL_i}{\p\!\stackrel{\rm iv}{x}{\!}^j} & \doteq
\frac2{|a|^4}\epsilon_{pjr} u^p a_i b^r - \frac8{|a|^4}\epsilon_{pir} u^p a_j b^r 
- \frac2{|a|^4}\epsilon_{pqr} u^p a^q b^r (\delta_{ij} - u_i u_j) + \frac8{|a|^6}\epsilon_{pqr} u^p a^q b^r a_i a_j  \\ 
& - \frac{10}{|a|^4}\epsilon_{lij} u^l g(a,b) + \frac{16}{|a|^6}\epsilon_{pqj} a_i u^p a^q g(a,b) 
- \frac{24}{|a|^6}\epsilon_{pqi} a_j u^p a^q g(a,b) + \frac6{|a|^4}\epsilon_{pqi} u^p a^q b_j \\
& - \frac{4}{|a|^4}\epsilon_{pqj} u^p a^q b_i +\frac5{|a|^2}\epsilon_{kij} a^k 
+ \frac1{|a|^2}\epsilon_{pqj} u^p a^q u_i + \frac1{|a|^2}\epsilon_{pqi} u^p a^q u_j.  
    \end{split}
\end{equation}

We must now now check that this coefficient vanishes for all $i,j = 1,2,3$. 
One may indeed, guess it should be so, as no terms with $\p_k\Gamma_{ij}^l$ factor appears 
(these are responsible for curvature), hence the coefficients looks precisely as in the flat case.
Let us (using isometry invariance) rotate the space so that $u = (1,0,0), a = (0,|a|,0)$ and $b = (-|a|^2, b^2, b^3)$.
Then we have:

 \begin{itemize}
\item $(i,j)=(1,1)$: $\frac{\p EL_1}{\p\!\stackrel{\rm iv}{x}{\!}^1}(x_0)=0$.
  \medskip
\item $(i,j)=(1,2)$: $\frac{\p EL_1}{\p\!\stackrel{\rm iv}{x}{\!}^2}(x_0)=0$.
  \medskip
\item $(i,j)=(1,3)$: $\frac{\p EL_1}{\p\!\stackrel{\rm iv}{x}{\!}^3}(x_0)=
	\frac6{|a|^4}\epsilon_{123} u^1 a^2 b^1+\frac5{|a|^2}\epsilon_{213} a^2 +\frac1{|a|^2}\epsilon_{123} u^1a^2
	=\frac4{|a|} - \frac5{|a|} + \frac1{|a|} = 0$.
  \medskip
\item $(i,j)=(2,1)$: $\frac{\p EL_2}{\p\!\stackrel{\rm iv}{x}{\!}^1}(x_0)=0$.
  \medskip
\item $(i,j)=(2,2)$: $\frac{\p EL_2}{\p\!\stackrel{\rm iv}{x}{\!}^2}(x_0)= -\frac6{|a|^3} b^3 -\frac2{|a|^3} b^3 +\frac8{|a|^3} b^3= 0$.
  \medskip
\item $(i,j)=(2,3)$: $\frac{\p EL_2}{\p\!\stackrel{\rm iv}{x}{\!}^3}(x_0)= -\frac2{|a|^3} b^2 +\frac{16}{|a|^3} b^2 -\frac{10}{|a|^3} b^2  -\frac4{|a|^3} b^2 = 0$.
  \medskip
\item $(i,j)=(3,1)$: $\frac{\p EL_3}{\p\!\stackrel{\rm iv}{x}{\!}^1}(x_0)= -\frac6{|a|} + \frac5{|a|} + \frac1{|a|} = 0$.
  \medskip
\item $(i,j)=(3,2)$: $\frac{\p EL_3}{\p\!\stackrel{\rm iv}{x}{\!}^2}(x_0)= \frac8{|a|^3} b^2 +\frac{10}{|a|^3} b^2 -\frac{24}{|a|^3} b^2 +\frac6{|a|^3} b^2 = 0$.
  \medskip
\item $(i,j)=(3,3)$: $\frac{\p EL_3}{\p\!\stackrel{\rm iv}{x}{\!}^3}(x_0)= -\frac2{|a|^3} b^3 +\frac6{|a|^3} b^3 -\frac4{|a|^3} b^3 = 0$.
 \end{itemize}

Since $x_0$ was arbitrary, we conclude that the Euler--Lagrange equations are of third-order.
 \end{proof}

 \begin{proof}[{\bf Proof of Theorem \ref{Thm}}]
By Lemmata \ref{ord5} and \ref{ord4} the Euler--Lagrange equation $\op{EL}=\op{EL}(L)$ for $L$ given by \eqref{LL} 
has order $\leq3$. Because of this, it can be written as 
 \begin{equation}\label{ELeq3}
\op{EL}_i=\frac{\p L}{\p x^i}-D_t\frac{\p L}{\p\dot{x}^i}
+D_t^2\frac{\p L}{\p\ddot{x}^i}-D_t^3\frac{\p L}{\p\dddot{x}^i}, 
 \end{equation}
where $D_t=\dot{x}^i\p_{x^i}+\ddot{x}^i\p_{\dot{x}^i}+\dddot{x}^i\p_{\ddot{x}^i}$ is the truncated total derivative.
The difference with \eqref{ELord3} is that $D_t$ is a vector field on the space $J^3$ of 3-jets, while
$\frac{d}{dt}$ is a vector field on the space $J^\infty$ of all jets.
 
Equation of conformal geodesics \eqref{BE0} also has order 3 with all 3-jets expressed via lower order jets.

 \begin{lemma}
Evaluation of $\op{EL}$ on \eqref{BE0} is identically zero.
 \end{lemma}
 
 \begin{proof}
We explain two verifications of this claim. First note that $\op{EL}$ is quadratic in 3-jets. Indeed, because at most 4th derivative
of $L$ appears in \eqref{ELeq3}, it is easy to see that degree of $\op{EL}$ in $\dddot{x}$ is at most 3. 
Furthermore, from formula \eqref{ELeq3} we observe that the top degree part of $\op{EL}_k$ corresponds to
 $$
\dddot{x}^i\dddot{x}^j\frac{\p^3L}{\p\ddot{x}^i\p\ddot{x}^j\p\ddot{x}^k}-
\dddot{x}^i\dddot{x}^j\dddot{x}^l\frac{\p^4L}{\p\ddot{x}^i\p\ddot{x}^j\p\ddot{x}^l\p\dddot{x}^k}.
 $$
Actually 
 $$
\frac{\p^3\op{EL}_k}{\p\dddot{x}^i\p\dddot{x}^j\p\dddot{x}^l}=
\frac{\p^3\lambda_l}{\p\ddot{x}^i\p\ddot{x}^j\p\ddot{x}^k}-\frac{\p^3\lambda_k}{\p\ddot{x}^i\p\ddot{x}^j\p\ddot{x}^l}=0
 $$
by \eqref{eq:EL2}, where $\lambda_s$ are coefficients of the third jets of $L$ as in \eqref{eq:EL1}.
In fact, for our Lagrangian $\lambda_s=\frac{\ell}{A^2}\epsilon_{pqs}u^pa^q$ and the required property was 
demonstrated in the proof of Proposition \ref{ord5}.
 
Thus it remains to verify vanishing of $\dddot{x}$-coefficients of degree 0,1,2 in $\op{EL}$. This is done
in a way similar to the proof of Lemma \ref{ord4}; this computation is even longer and we omit it.

Alternatively, one may pass from $(x,\dot{x},\ddot{x},\dddot{x})$ coordinates to $(x,u,a,b)$ coordinates on $J^3$.
With the inverse map denoted $\phi$, the corresponding transformation of basic vector fields on $J^3$ is
 \begin{gather*}
\phi_*\p_{b^j}=\p_{\dddot{x}^j},\qquad
\phi_*\p_{a^j}=\p_{\ddot{x}^j}-3\Gamma^k_{ij}\dot{x}^i\p_{\ddot{x}^k},\\
\phi_*\p_{u^j}=\p_{\dot{x}^j}-2\Gamma_{ij}^k\dot{x}^i\p_{\ddot{x}^k}-3\Bigl(
 \Gamma_{ij}^k\ddot{x}^i+S_{(ijl)}^k\dot{x}^i\dot{x}^l-2\Gamma_{pq}^k\Gamma^p_{ij}\dot{x}^i\dot{x}^q\Bigr)
 \p_{\dddot{x}^k},\\
\phi_*\p_{x^j}=\p_{x^j}-\p_j\Gamma_{pq}^k\dot{x}^p\dot{x}^q\p_{\ddot{x}^k}-\Bigl(
 3\p_j\Gamma_{pq}^k\dot{x}^p\ddot{x}^q+\p_jS_{ipq}^k\dot{x}^i\dot{x}^p\dot{x}^q-3\p_j\Gamma_{is}^k\Gamma^s_{pq}
 \dot{x}^i\dot{x}^p\dot{x}^q\Bigr)\p_{\dddot{x}^k}.
 \end{gather*}
While the Lagrangian \eqref{LL} is simpler in $(x,u,a,b)$ coordinates, the truncated total derivative $D_t$ is more
complicated, and this calculation is of the same complexity.

Another verification is computer based. In fact, since the computation is purely symbolic and only relies on integer arithmetic,
the rigor does not suffer with computer algebra approach. We did it in the symbolic package Maple. 
To simplify the computation, we assumed that the metric $g$ is in diagonal coordinates, which is always possible 
to achieve locally in 3D. 
Also in verification of vanishing we substituted constraints $|u|^2=1$, $g(u,a)=0$, etc, as in the proof of Lemma \ref{ord4}.
Our program is accessible as an ancillary file in the arXiv version, and it shows identical cancellation of all terms upon 
substitution of \eqref{BE0} into $\op{EL}$. 
 \end{proof}
 
Thus conformal geodesic equations imply the Euler--Lagrange equations for $L$.

For the reverse direction, note that $\op{rank}(A)=2$, where $A$ is the $3\times3$ matrix $[\p\op{EL}_i/\p\dddot{x}^j]$.
We already discussed that the rank of $A$ cannot exceed 2, and it is equal to 2 for the flat metric
(and hence any metric $C^2$ close to it). In general, the required equality is obtained by a straightforward computation.

Since the rank of the equation for unparametrized conformal geodesics is also 2, these equations are equivalent.
Therefore $\op{EL}$ is the equation of unparametrized conformal geodesics, as was claimed.
 \end{proof}

Finally, let us indicate a Lagrangian of the second order that also solves the variational problem.
Similarly to the flat case, let us note that 
 $$
\lambda_s=\frac{\p L}{\p b^s}=\frac{\ell}{A^2}\epsilon_{ijs}u^ia^j 
 $$
is the gradient by $a$, i.e.\ $\lambda_s=\frac{\p\lambda}{\p a^s}$ for some $\lambda$. 
This potential has the form (note the lowered indices)
 $$
\lambda=\op{arctan}\frac{a_1\ell^2-g(u,a)u_1}{\ell\sqrt{|g|}(\epsilon_{1ij}u^ia^j)}
 $$
and, similar to \eqref{2ndordL}, passing to  
 $$
L'=L-\frac{d\lambda}{dt} 
 $$
we get a second order Lagrangian (affine in second jets) with precisely unparametrized conformal geodesics as extremals. 
This $L'$ is neither covariant nor isometry invariant.

\section{On conformal invariance}\label{S6}

Since the equation for conformal geodesics is conformally invariant, it is a natural question whether 
it is given by a conformally invariant Lagrangian. Several attempts were made to identify such.

In \cite{BE} the following Lagrangian of order 2 (nondegenerate in 2-jet variable $a$) was considered 
 $$
L=\frac{|a|^2}{|u|^2}-2\frac{(u\cdot a)^2}{|u|^4}+2P(u,u).
 $$
It is conformally invariant up to divergence, but the corresponding Euler--Lagrange equation is of order 4. 
Thus it contains more extremals than just conformal geodesics (the latter are extremals under 
an additional condition $K=0$, where $K$ is a certain conformally invariant vector, vanishing in the conformally flat case).

In \cite{DK} the Lagrangian $L=\frac{u\cdot E}{|u|^2}$ of order 3 was considered, where $E$ is equal to $b=\nabla_ua$
minus the expression in the right-hand side of \eqref{BE0} (in other words, components of the vector $E$ give the equations of 
parametrized conformal geodesics). 
This Lagrangian is conformally invariant, but again its extremals are more plentiful than just conformal geodesics;
see also \cite{SZ} for similar results. 

Our Lagrangian $L$ is covariant and even invariant with respect to isometries, but
it is not conformally invariant. However we have:
 \begin{theorem}
Lagrangian \eqref{LL} is conformally invariant up to divergence. More precisely, if $L$ is the 
Lagrangian \eqref{LL} computed for the metric $g$ and $\bar{L}$ is such for $\bar{g}=e^{-2\f}g$, then
 \begin{equation}\label{LbarL}
\bar{L}-L=\frac{d}{dt}S,\quad \text{where}\quad 
S=\op{arccos}\Bigl(\frac{g(\pi(a),\pi(\bar{a}))}{|\pi(a)|_g\cdot|\pi(\bar{a})|_g}\Bigr)
 \end{equation}
and $a=\nabla_uu$, $\bar{a}=\bar\nabla_{\bar{u}}\bar{u}$, $\bar{u}=e^\f u$ and $\pi:TM\to u^\perp$,
$\pi(v)=v-\tfrac{g(v,u)}{|u|_g^2}u$ is the orthogonal projection.
 \end{theorem}

This conformal invariance is equivalent to the claim that $\int L\,dt$ has identically same Euler--Lagrange equations as
its conformally related action $\int\bar{L}\,dt$. Invariance of finite torsion functional was proved by other methods 
in \cite[Proposition 5.11]{TM}. Our result contains a precise formula for the transformation of the Lagrangian, and
as a by-product the transformation of the torsion under conformal change of the metric.

 \begin{proof}
We will compare two Lagrangians \eqref{LL} for $g$ and $\bar{g}=e^{-2\f}g$. 
Recall that under conformal rescaling the Levi-Civita connection changes as follows:
 $$
\bar\nabla_XY= \nabla_XY-X(\f)Y-Y(\f)X+g(X,Y)\nabla^g\f.
 $$
Let $u=\dot{x}$ be the unit vector for $g$, tangent to $\gamma$, then the corresponding unit vector for 
$\bar{g}$ is $\bar{u}=e^\f u$.

For the acceleration $a=\nabla_uu$ we have
 $$
\bar{a}=\bar\nabla_{\bar{u}}\bar{u}= e^{2\f}\bigl(a-\p_u(\f)u+\nabla^g\f\bigr).
 $$
Note that $g(u,\bar{a})=0$ (we will use $g(u,u)=1$, $g(u,a)=0$ throughout). 
 
Next we compute $b=\nabla_ua$ and 
 \begin{align*}
\bar{b} &= \bar\nabla_{\bar{u}}\bar{a}=e^\f\bigl(\nabla_u\bar{a}-\p_u(\f)\bar{a}-\p_{\bar{a}}(\f)u\bigr)\\
&= e^{3\f}\bigl(b-\p_u(\f)^2u+\p_u(\f)\nabla^g\f-\p_a(\f)u+\nabla_u\nabla^g\f-|\nabla^g\f|_g^2u\bigr).
 \end{align*} 

From this we get the formula for the volume $V=\op{Vol}_g(u,a,b)$ using its skew-symmetry:
 \begin{align*}
\bar{V} &= \op{Vol}_{\bar{g}}(\bar{u},\bar{a},\bar{b})=
e^{3\f}\op{Vol}_g(u,a+\nabla^g\f,b+\p_u(\f)\nabla^g\f+\nabla_u\nabla^g\f)\\
&= e^{3\f}\bigl(V+\op{Vol}_g(u,\nabla^g\f,b)+\p_u(\f)\op{Vol}_g(u,a,\nabla^g\f)
+\op{Vol}_g(u,a,\nabla_u\nabla^g\f)+\op{Vol}_g(u,\nabla^g\f,\nabla_u\nabla^g\f)\bigr),
 \end{align*} 
and for the area square $A^2=g(a,a)$ we get the following:
 $$
\bar{A}^2=\bar{g}(\bar{a},\bar{a})=e^{2\f}\bigl(A^2+2\p_a(\f)-\p_u(\f)^2+|\nabla^g\f|_g^2\bigr).
 $$
Thus (note that $d\bar{t}=e^{-\f}dt$) we obtain:
 $$
\bar{L}\,d\bar{t}=\frac{V\bigl(1+\p_a(\f)+\p_u(\f)\p_b(\f)\bigr)+\op{Vol}_g(u,a+\nabla^g\f,\nabla_u\nabla^g\f)}
{A^2+2\p_a(\f)-\p_u(\f)^2+|\nabla^g\f|_g^2}\,dt.
 $$
This expression is quasilinear in the second jets of $\f$ and 
from the symbol in $\f$ we conclude that $L$ is not conformally invariant.

To see that $\bar{L}-L$ is a divergence, one may first concentrate on the symbol (top derivatives) in $\f$.
It may be identified with 1-form $\sigma=\sum_{i=1}^3S_id\f_i$, where $\nabla^g\f=(\f_i)$. 
Checking compatibility $d\sigma=0$ is equivalent to identity \eqref{eq_01} estabilished in the proof of Lemma \ref{ord5}.
This allows to obtain the potential $S$ via $S_i=S_{\f_i}$.

In order to verify \eqref{LbarL} we choose geodesic coordinates for $g$ at the point $x_0\in\gamma\subset M$, where all
expressions simplify and the identiy follows by a straightforward computation.
 \end{proof}

\bigskip

Denoting by $\measuredangle(v,w)=\op{arccos}\Bigl(\frac{g(v,w)}{|v|_g\cdot|w|_g}\Bigr)$ the angle between vectors $v$ and $w$,
we derive from \eqref{LbarL} the following relation for any vector $v$ (in fact only the direction of $v$ is important in this formula):
 \begin{equation}\label{LbarL2}
\bar{L}-\frac{d}{dt}\measuredangle(\pi(\bar{a}),v)=L-\frac{d}{dt}\measuredangle(\pi(a),v).
 \end{equation}
 
The theory of global differential invariants \cite{KL} implies the existence of conformally invariant vector
field $\xi$ depending on 4-jet of the curve, complementary to the direction of $u$.
A particular form of $\xi$ is not important here and so is omitted, for instance we refer to the conformal Frenet frames in \cite{F}. 

With such a choice, due to \eqref{LbarL2} the following Lagrangian has the same extremals as $L$ and 
(the corresponding form $L\,dt$) is fully conformally invariant:
\begin{equation}\label{LbarL3}
\hat{L}=L-\frac{d}{dt}\measuredangle(\pi(a),\xi).
 \end{equation}
Reference \cite{M} operates with similar ideas, but \cite{MMR} questions its ``magical identity involving a globally defined angle''.
We interpret this angle as our $\measuredangle(\cdot,\xi)$.

The conformally invariant Lagrangian form given by \eqref{LbarL3} 
was already discovered in the conformally flat case as the conformal torsion $T$
times the conformal arc length element $\omega$, see \cite{M,CSW} (in other sources, $T$ is denoted by
$k_2$ and $\mu_2$). The corresponding functional (taken with values in $S^1=\R\,\op{mod}2\pi$)
 \[
\int T\,\omega=\int \hat{L}\,dt=\int L\,dt=\int \tau\,ds
 \]
is known as the total twist or the Banchoff--White functional (its conformal invariance was proven in \cite{BW}). 
The expressions for the conformal torsion $T$ and the arc length element $\omega$ through 
the metric curvature $\kappa$, the torsion $\tau$ and the natural parameter $ds$ 
of a curve in the Euclidean space $\R^3$ are the following:
 \[
T=\Bigl(2\kappa_s^2\tau + \kappa\,\kappa_s\tau_s - \kappa\,\kappa_{ss}\tau + \kappa^2\tau^3\Bigr)\Bigl/
\sqrt[4]{(\kappa_s^2+\kappa^2\tau^2)^5},\quad
\omega=\sqrt[4]{\kappa_s^2+\kappa^2\tau^2}\,ds.
 \] 

Thus we conclude that raising the order of Lagrangian to 4
yields a possibility to choose a conformally invariant Lagrangian $\hat{L}$ for conformal geodesics.
Apparently such a Lagrangian $L$ cannot exist in order 3. 

 \begin{rk}
The nomenclature ``conformal geodesics'' was used differently in \cite{M,MMR}, namely as the Euler--Lagrange equation
for the conformal arc length $\int\omega$. This is a sixth order equation, and it is variational by definition.
We however keep using the terminolgy adapted in \cite{Y,FS,BE}.
 \end{rk}

Let us conclude by mentioning that it may be interesting to see if the invariant variational bicomplex technique \cite{G,An} 
would work for the invariant variational problem $\int\hat{L}dt=\int T\,\omega$. 
This question will be addressed in the forthcoming paper \cite{KSS}.

\bigskip

{\bf Acknowledgment.}
The research of B.K.\ and W.S.\ was partially supported by
the Tromsø Research Foundation (project “Pure Mathematics in Norway”) and
the UiT Aurora project MASCOT.
V.M. thanks the DFG (projects 455806247 and 529233771), and the
ARC Discovery Programme DP210100951 for their support.
B.K.\ and V.M.\ also acknowledge DAAD-RCN collaboration (PPP  57525061) between Germany and Norway, 
during which they first discussed the problem of focus in this work.
We thank K.\ Neusser, J.~Silhan and V.\ Zadnik for useful discussions.

\end{document}